\documentclass{birkjour}
\usepackage{amsfonts,
amsmath,
amsthm,
amscd,
amssymb,
latexsym,cite,verbatim,texdraw,floatflt,pb-diagram}

\usepackage{url}
\usepackage{comment}
\usepackage{amsxtra}
\usepackage{latexsym}

\usepackage{tikz} 

 \newtheorem{thm}{Theorem}[section]
 \newtheorem{cor}[thm]{Corollary}
 \newtheorem{lem}[thm]{Lemma}
 \newtheorem{prop}[thm]{Proposition}
 \theoremstyle{definition}
 \newtheorem{defn}[thm]{Definition}
 \theoremstyle{remark}
 \newtheorem{rem}[thm]{Remark}
 \newtheorem{ex}{Example}
 \numberwithin{equation}{section}

\newcommand{\ve}{\varepsilon}
\newcommand{\RR}{\mathbb{R}}
\newcommand{\NN}{\mathbb{N}}
\begin{document}

%
%
%
%
%
%
%
%
%

\title[On generalizations of some fixed point theorems]{On generalizations of some fixed point theorems in semimetric spaces with triangle functions \footnote{This is the submitted version. The published version is available at Front. Appl. Math. Stat., 17 June 2024,
Sec. Numerical Analysis and Scientific Computation, Volume 10, 2024. \url{https://doi.org/10.3389/fams.2024.1392560}}}

\author[Evgeniy Petrov]{Evgeniy Petrov}

\address{
Institute of Applied Mathematics and Mechanics\\
of the NAS of Ukraine\\
Slovyansk,Ukraine}

\email{eugeniy.petrov@gmail.com}


\author[Ruslan Salimov]{Ruslan Salimov}

\address{
Institute of Mathematics\\
of the NAS of Ukraine\\
Kiev, Ukraine}

\email{ruslan.salimov1@gmail.com}


\author{Ravindra K. Bisht}
\address{Department of Mathematics\\
National Defence Academy\\
411023 Khadakwasla\\
Pune, India}

\email{ravindra.bisht@yahoo.com}






\begin{abstract}
In the present paper, we prove generalizations of Banach, Kannan, Chatterjea, \'Ciri\'c-Reich-Rus fixed point theorems, as well as of the fixed point theorem for mappings contracting perimeters of triangles. We consider corresponding mappings in semimetric spaces with triangle functions introduced by M. Bessenyei and Z. P\'ales. Such an approach allows us to derive corollaries for various types of semimetric spaces including metric spaces, ultrametric spaces, b-metric spaces etc.
The significance of these generalized theorems extends across multiple disciplines, including optimization, mathematical modeling, and computer science. They may serve to establish stability conditions, demonstrate the existence of optimal solutions, and improve algorithm design.
\end{abstract}

\maketitle

\section{Introduction}
The Contraction Mapping Principle was established by S. Banach in his dissertation (1920) and published in 1922~\cite{Ba22}. Although the idea of successive approximations in a number of concrete situations (solution of differential and integral equations, approximation theory) had appeared earlier in the works of P.~L.~Chebyshev, E.~Picard, R. Caccioppoli, and others, S. Banach was the first to formulate this result in a correct abstract form suitable for a wide range of applications.

In 1968, Kannan's pioneering work in fixed-point theory led to a significant result, independent of the Banach contraction principle \cite{Ka68}. Kannan's theorem provided a crucial characterization of metric completeness: A metric space $X$ is complete if and only if every mapping satisfying Kannan contraction on $X$ has a fixed point \cite{Su74}. This discovery spurred the introduction of numerous contractive definitions, many of which allowed for discontinuity in their domain. Among these contractive conditions, those explored by Chaterjee \cite{Ch72} and \'Ciri\'c-Reich-Rus \cite{CB71,Re71,Ru71} share similar characteristics, further enriching understanding of the properties of contractive mappings in metric spaces. For various contractive definitions, we suggest authors refer to a survey paper by Rhoades \cite{Rh77}.
After a century, the interest of mathematicians around the world in fixed point theorems remains high. This is evidenced by the appearance of numerous articles and monographs in recent decades dedicated to fixed point theory and its applications. For a survey of fixed point results and their diverse applications, see, for example, the monographs~\cite{Ki01, AJS18, Su18}.

Let $X$ be a nonempty set. Recall that a mapping  $d\colon X\times X\to \mathbb{R}^+$, $\mathbb{R}^+=[0,\infty)$ is a \emph{metric} if for all $x,y,z \in X$ the following axioms hold:
\begin{itemize}
  \item [(i)] $(d(x,y)=0)\Leftrightarrow (x=y)$,
  \item [(ii)] $d(x,y)=d(y,x)$,
  \item [(iii)] $d(x,y)\leqslant d(x,z)+d(z,y)$.
\end{itemize}
The pair $(X,d)$ is called a \emph{metric space}. If only axioms (i) and (ii) hold then $d$ is called a \emph{semimetric}. A pair $(X,d)$, where  $d$  is a semimetric on $X$, is called a \emph{semimetric space}. Such spaces were first examined by Fr\'{e}chet in~\cite{Fr06}, where he called them ``classes (E)''. Later these spaces and mappings on them attracted the attention of many mathematicians, see e.g.~\cite{Ch17,Ni27,Wi31,Fr37,DP13,PS22}.

In semimetric spaces, the notions of convergent and Cauchy sequences, as well as completeness, can be introduced in the usual way.\\

The concept of b-metric space was initially introduced by Bakhtin \cite{Ba89} under the name of quasi-metric spaces, wherein he demonstrated a contraction principle in this space. Czerwik \cite{C98,C93} further utilized such space to establish generalizations of Banach's fixed point theorem. In a b-metric space, the triangle inequality (iii) is extended to include the condition that there exists \( K \geq 1 \), ensuring that \( d(x, y) \leq K [d(x, z) + d(z, y)] \) for all \mbox{\( x, y, z \in X \)}. Fagin and Stockmeyer~\cite{FS98} further explored the relaxation of the triangle inequality within b-metric spaces, labeling this adjustment as nonlinear elastic matching (NEM). They observed its application across diverse domains, including trademark shape analysis \cite{CMV94}  and the measurement of ice floes \cite{Mc91} . Xia \cite{X09} utilized this semimetric distance to investigate optimal transport paths between probability measures.

Recall that an \emph{ultrametric} is a metric for which the strong triangle inequality $d(x, y)\leqslant \max \{d(x, z), d(z, y)\}$ holds for all \( x, y, z \in X \). In this case the pair $(X,d)$ is called an \emph{ultrametric space}. Note that, the ultrametric inequality was formulated by F.~Hausdorff in 1934 and ultrametric spaces were introduced by M. Krasner~\cite{Kr44} in 1944.

In 2017, Bessenyei and  P\'ales~\cite{BP17} extended the Matkowski fixed point theorem~\cite{Ma75} by introducing a definition of a triangle function $\Phi \colon \overline{\mathbb{R}}_{+}^2\to \overline{\mathbb{R}}^+$ for a semimetric $d$.  We adopt this definition in a slightly different form, restricting the domain and the range of $\Phi$ by ${\RR}_+^2$ and ${\RR}^+$, respectively.

\begin{defn}\label{d21}
Consider a semimetric space $(X, d)$. We say that $\Phi \colon {\RR}^+\times{\RR}^+ \to {\RR}^+$ is a \emph{triangle function} for $d$ if $\Phi$ is symmetric and {nondecreasing in both of its arguments}, satisfies $\Phi(0,0)=0$ and, for all $x, y, z \in X$, the generalized triangle inequality
\begin{equation}\label{tr}
d(x,y)\leqslant \Phi(d(x,z), d(z,y))
\end{equation}
holds.
\end{defn}

{Obviously, metric spaces, ultrametric spaces and $b$-metric spaces are semimetric spaces with the triangle functions $\Phi(u,v)=u+v$, $\Phi(u,v)=\max\{u,v\}$ and $\Phi(u,v)=K(u+v)$, respectively.}


In~\cite{BP17}, semimetric spaces with so-called basic triangle functions continuous at the origin were investigated. These spaces were termed regular. It was demonstrated that in a regular semimetric space, the topology is Hausdorff, a convergent sequence has a unique limit, and possesses the Cauchy property, among other properties. For further developments in this area, see also~\cite{JT20,VH17,CJT18,BP22,KP22}.\\

In this paper, we revisit several well-known fixed-point theorems, either extending their capabilities by modifying their assumptions or presenting new and innovative proofs. With the help of key Lemma 1.2 and its conclusion, we unveil further results that offer insightful perspectives on the nature of fixed-point theorems, not only within the metric context but also within more general spaces.

Here is the key lemma essential for the subsequent sections.
\begin{lem}\label{lem}
Let $(X,d)$ be a semimetric space with the triangle function $\Phi$
satisfying the following conditions:
\begin{itemize}
\item[1)] The equality
\begin{equation}\label{ee1}
  \Phi(ku,kv) = k\Phi(u,v)
\end{equation}
holds for all $k,u,v\in \RR^+$.
\item[2)] For every $0\leqslant \alpha <1$ there exists $C(\alpha)>0$ such that for every
$p\in \NN^+$ the inequality
\begin{equation}\label{ee2}
  \Phi(1,\Phi(\alpha, \Phi(\alpha^2,....,\Phi(\alpha^{p-1},\alpha^{p}))))\leqslant C(\alpha)
\end{equation}
holds.
\end{itemize}

Let $(x_n)$, $n=0,1,\ldots$, be a sequence in $X$ having the property that there exists $\alpha\in [0,1)$ such that
\begin{equation}\label{e31n0}
d(x_n,x_{n+1})\leqslant \alpha d(x_{n-1},x_{n})
\end{equation}
for all $n\geqslant 1$. Then $(x_n)$ is a Cauchy sequence.
\end{lem}

\begin{proof} We break the proof of this lemma into several parts.
1. {Initial bounds}: by~(\ref{e31n0}) we have
$$d(x_1,x_2)\leqslant \alpha d(x_0,x_1), \,\,
d(x_2,x_3)\leqslant \alpha d(x_1,x_2), \,\,
d(x_3,x_4)\leqslant \alpha d(x_2,x_3), \,\,\ldots.$$
Hence, we obtain
\begin{equation}\label{e31n}
d(x_n,x_{n+1})\leqslant \alpha^n d(x_0,x_1).
\end{equation}

2. Use of generalized triangle inequality ~(\ref{tr}):
applying consecutively generalized triangle inequality~(\ref{tr}) to the points $x_n$, $x_{n+1}$, $x_{n+2}$, \ldots, $x_{n+p}$, where $p\in \NN^+$, $p\geqslant 2$, we obtain
$$
d(x_n,\,x_{n+p})\leqslant \Phi(d(x_{n},\,x_{n+1}), d(x_{n+1},\,x_{n+p}))
$$
$$
\leqslant \Phi(d(x_{n},\,x_{n+1}),  \Phi(d(x_{n+1},\,x_{n+2}), d(x_{n+2},\,x_{n+p})))
$$
$$
\ldots
$$
\begin{equation*}
\leqslant \Phi(d(x_{n},\,x_{n+1}),\Phi(d(x_{n+1},\,x_{n+2}),\ldots, \Phi(d(x_{n+p-2},\,x_{n+p-1}), d(x_{n+p-1},\,x_{n+p})))).
\end{equation*}

3. {Utilizing properties of $\Phi$}: by the monotonicity of $\Phi$ and inequalities (\ref{e31n}), we have

$$
d(x_n,\,x_{n+p})\leqslant \Phi(\alpha^{n}d(x_0,x_1),\Phi(\alpha^{n+1}d(x_0,x_1),\cdots,
$$
$$
\Phi(\alpha^{n+p-2}d(x_0,x_1),\alpha^{n+p-1}d(x_0,x_1)))).
$$
Applying several times equality~(\ref{ee1}), we get
$$
d(x_n,\,x_{n+p})\leqslant
\alpha^{n}\Phi(1,\Phi(\alpha,\cdots,\Phi(\alpha^{p-2},\alpha^{p-1})))d(x_0,x_1).
$$

4. {Bounding the expression and concluding Cauchy sequence}: by condition~(\ref{ee2}), we obtain
\begin{equation}\label{ee3}
d(x_n,\,x_{n+p})\leqslant \alpha^{n}C(\alpha)d(x_0,x_1).
\end{equation}
Since $0\leqslant\alpha<1$, we have $d(x_n,\,x_{n+p})\to 0$ as $n\to \infty$ for every $p\geqslant 2$. If $p=1$, then the relation $d(x_n,\,x_{n+1})\to 0$ follows from~(\ref{e31n}). Thus, $(x_n)$ is a Cauchy sequence, which completes the proof.
\end{proof}

\begin{rem}
Let $(X,d)$ be a complete semimetric space, then the sequence $(x_n)$ has a limit $x^*$. If additionally the semimetric $d$ is continuous, then we get $d(x_n,x_{n+p})\to d(x_n, x^*)$ as $p\to \infty$. Hence, letting $p\to \infty$ in~(\ref{ee3}) we get
\begin{equation}\label{ee5}
d(x_n,\,x^*)\leqslant \alpha^{n}C(\alpha)d(x_0,x_1).
\end{equation}
\end{rem}


\section{Banach contraction principle in semimetric spaces}

It is possible to extend the well-known concept of contraction mapping to the case of semimetric spaces. We shall say that a mapping $T\colon X\to X$ is a \emph{contraction mapping} on the semimetric space $(X,d)$ if there exists $\alpha\in [0,1)$ such that
\begin{equation}\label{e3}
d(Tx,Ty)\leqslant \alpha d(x,y)
\end{equation}
for all $x,y \in X$.

\begin{thm}\label{t2}
Let $(X,d)$ be a complete semimetric space with the triangle function $\Phi$ continuous at $(0,0)$ and satisfying conditions~(\ref{ee1}) and~(\ref{ee2}). Let  $T\colon X\to X$ be a contraction mapping. Then $T$ has a unique fixed point.
\end{thm}

\begin{proof}
Let $x_0 \in X$ and let $x_n=Tx_{n-1}$, $n=1,2,...$. By~(\ref{e3}) and by Lemma~\ref{lem} $(x_n)$ is a Cauchy sequence and  by completeness of $(X,d)$, this sequence has a limit $x^*\in X$.

Let us prove that $Tx^*=x^*$. It is easy to see that the contraction mappings on semimetric spaces are continuous. {Indeed, let $y_n\to y_0$ as $n\to \infty$, then $d(y_n,y_0)\to 0$ and by~(\ref{e3}) we have $d(Ty_n,Ty_0)\to 0$, i.e., $Ty_n\to Ty_0$.} Since $x_n\to x^*$, by the continuity of $T$ we have $x_{n+1}=T x_n\to Tx^*$.  By triangle inequality~(\ref{tr}) and continuity of $\Phi$ at $(0,0)$ we have
$$
d(x^*,Tx^*)\leqslant \Phi(d(x^*,x_{n}),d(x_{n},Tx^*))
\to 0
$$
as $n\to \infty$, which means that $x^*$ is the fixed point.

Suppose that there exist two distinct fixed points $x$ and $y$. Then $Tx=x$ and $Ty=y$, which contradicts to~(\ref{e3}).
\end{proof}

\begin{cor}\label{c33}
The following assertions hold:
\begin{itemize}
  \item [(i)] (\textbf{Banach contraction principle}) Theorem~\ref{t2} holds for metric spaces, i.e., for semimetric spaces with the triangle function $\Phi(u,v)=u+v$.
  \item [(ii)] The following inequality holds:
  $$
  d(x_n,x^*)\leqslant \frac{\alpha^n}{1-\alpha}d(x_0,x_1).
  $$
\end{itemize}
\end{cor}
\begin{proof}
\textbf{(i)} It is easy to see that $\Phi$ satisfies equality~(\ref{ee1}) and $\Phi$ is continuous at $(0,0)$. Consider expression~(\ref{ee2}) for such power triangle functions $\Phi$:
$$
1+\alpha+\alpha^{2}+\cdots+\alpha^{(p-1)}+\alpha^{p}.
$$
According to the formula for the sum of infinite geometric series this sum is less than $1/(1-\alpha)=C(\alpha)$ for every finite $p\in \NN^+$, which establishes inequality~(\ref{ee2}).

Assertion \textbf{(ii)} follows directly from~(\ref{ee5}).
\end{proof}

\begin{cor}\label{c34}
The following assertions hold:
\begin{itemize}
  \item [(i)] Theorem~\ref{t2} holds for ultrametric spaces, i.e., for semimetric spaces with the triangle function $\Phi(u,v)=\max\{u,v\}$.
  \item [(ii)] The following inequality holds:
  $$
  d(x_n,x^*)\leqslant \alpha^n d(x_0,x_1).
  $$
\end{itemize}
\end{cor}
\begin{proof}
\textbf{(i)} It is easy to see that $\Phi$ satisfies equality~(\ref{ee1}) and $\Phi$ is continuous at $(0,0)$. Consider expression~(\ref{ee2}) for the power triangle functions $\Phi$. Since $\alpha<1$ we have
$$
\max\{1,\alpha,\alpha^{2},\cdots,\alpha^{p-1},\alpha^{p}\}=1=C(\alpha).
$$
which establishes inequality~(\ref{ee2}).

Assertion \textbf{(ii)} follows directly from~(\ref{ee5}).
\end{proof}

Distance spaces with power triangle functions $\Phi(u,v)=(u^q+v^q)^{\frac{1}{q}}$, $q\in[-\infty, \infty]$ were considered in~\cite{Gr16}. In~\cite{Gr16} these functions have a little more general form. Note also that semimetric spaces with power triangle functions are metric spaces if $q\geqslant 1$.

\begin{cor}\label{c35}
The following assertions hold:
\begin{itemize}
  \item [(i)] Theorem~\ref{t2} holds for semimetric spaces with power triangle functions $\Phi(u,v)=(u^q+v^q)^{\frac{1}{q}}$ if $q>0$.
  \item [(ii)] The following inequality holds for $q\geqslant 1$:
  $$
  d(x_n,x^*)\leqslant \frac{\alpha^n}{(1-\alpha^q)^{\frac{1}{q}}}d(x_0,x_1).
  $$
\end{itemize}
\end{cor}

\begin{proof}
\textbf{(i)} It is easy to see that $\Phi$ satisfies equality~(\ref{ee1}) and $\Phi$ is continuous at $(0,0)$. Consider expression~(\ref{ee2}) for the power triangle functions $\Phi$:
$$
(1+\alpha^q+\alpha^{2q}+\cdots+\alpha^{(p-1)q}+\alpha^{pq})^{\frac{1}{q}}.
$$
It is clear that the sum
\begin{equation}\label{e49}
1+\alpha^q+\alpha^{2q}+\cdots+\alpha^{(p-1)q}+\alpha^{pq}
\end{equation}
consists of $p+1$ terms of geometric progression with the common ratio $\alpha^q$ and start value $1$. Since $\alpha<1$ we have the inequality $\alpha^q<1$. According to the formula for the sum of infinite geometric series sum~(\ref{e49}) is less than $1/(1-\alpha^q)$ for every finite $p\in \NN^+$.
Hence,
$$
(1+\alpha^q+\alpha^{2q}+\cdots+\alpha^{(p-1)q}+\alpha^{pq})^{\frac{1}{q}}
< (1/(1-\alpha^q))^{\frac{1}{q}}=C(\alpha),
$$
which establishes inequality~(\ref{ee2}).

Assertion \textbf{(ii)} follows directly from~(\ref{ee5}) and from the fact that semimetric spaces with power triangle functions are metric spaces if $q\geqslant 1$.
\end{proof}


\begin{cor}\label{c36}
 Theorem~\ref{t2} holds for $b$-metric spaces {with the coefficient $K$ } if $\alpha K<1$, where $\alpha$ is the coefficient from~(\ref{e3}).
\end{cor}

\begin{proof}
It is clear that $\Phi(u,v)=K(u+v)$ satisfies condition~(\ref{ee1}) and it is continuous at $(0,0)$.
Consider expression~(\ref{ee2}) for the function $\Phi$:

\begin{equation}\label{e519}
K+K^2\alpha+K^3\alpha^{2}+\cdots+K^p\alpha^{p-1}+K^{p}\alpha^{p}
\end{equation}
$$
\leqslant K+K^2\alpha+K^3\alpha^{2}+\cdots+K^p\alpha^{p-1}+K^{p+1}\alpha^{p}.
$$
It is clear that this sum consists of $p+1$ terms of geometric progression with the common ratio $\alpha K$ and the start value $K$. According to the formula for the sum of infinite geometric series sum~(\ref{e519}) is less than $K/(1-\alpha K)=C(\alpha)$ for every finite $p\in \NN^+$, which establishes inequality~(\ref{ee2}).
\end{proof}

Note that Corollary~\ref{c36} is already known, see Theorem 1 in~\cite{KK13}.

\section{Kannan's contractions in semimetric spaces}

In~\cite{Ka68} Kannan proved the following result which gives the fixed point for discontinuous mappings.
\begin{thm}
Let $T\colon X\to X$ be a mapping on a complete metric space $(X,d)$ such that
  \begin{equation}\label{kk}
   d(Tx,Ty)\leqslant \beta (d(x,Tx)+d(y,Ty)),
  \end{equation}
where $0\leqslant \beta<\frac{1}{2}$ and $x,y \in X$. Then $T$ has a unique fixed point.
\end{thm}
The mappings satisfying inequality~(\ref{kk}) are called \emph{Kannan type mappings}. 

\begin{thm}\label{t7}
Let $(X,d)$ be a complete semimetric space with the continuous triangle function $\Phi$, satisfying conditions~(\ref{ee1}) and~(\ref{ee2}). Let  $T\colon X\to X$ satisfy inequality~(\ref{kk}) with some $0\leqslant \beta<\frac{1}{2}$ and let additionally the following condition hold:
\begin{itemize}
\item [(i)]  $\Phi(0,\beta)<1$.
\end{itemize}
 Then $T$ has a unique fixed point.
\end{thm}

\begin{proof}
Let $x_0\in X$. Define $x_n = T^n x_0$ for $n = 1, 2, \ldots$. It follows straightforwardly that
$$
d(x_n,x_{n+1}) = d(Tx_{n-1},Tx_n)
$$
$$
\leqslant \beta (d(x_{n-1},Tx_{n-1}) + d(x_n,Tx_n)) = \beta (d(x_{n-1},x_n) + d(x_n,x_{n+1})),
$$
and
\[
d(x_n,x_{n+1}) \leqslant \alpha d(x_{n-1},x_n),
\]
where $\alpha = \frac{\beta}{1-\beta}$, $0\leqslant \alpha <1$. By Lemma~\ref{lem} $(x_n)$ is a Cauchy sequence and  by completeness of $(X,d)$, this sequence has a limit $x^*\in X$.

Let us prove that $Tx^*=x^*$. 
By triangle inequality~(\ref{tr}), the monotonicity of $\Phi$ and~(\ref{kk}) we get
$$
d(x^*,Tx^*)\leqslant \Phi(d(x^*,T^nx_{0}),d(T^nx_{0},Tx^*))
$$
$$
\leqslant \Phi(d(x^*,T^nx_{0}),\beta (d(T^{n-1}x_{0},T^{n}x_{0})+d(x^*,Tx^*))).
$$
Letting $n\to \infty$, by the continuity of $\Phi$  we obtain
$$
d(x^*,Tx^*)\leqslant \Phi(0,\beta d(x^*,Tx^*)).
$$
{Using~(\ref{ee1}), we have
$$
d(x^*,Tx^*)\leqslant d(x^*,Tx^*)\Phi(0,\beta).
$$}
By condition (i) we get $d(x^*,Tx^*)=0.$

Suppose that there exist two distinct fixed points $x$ and $y$. Then $Tx=x$ and $Ty=y$, which contradicts to~(\ref{kk}).
\end{proof}

\begin{cor}\label{c334}
Theorem~\ref{t7} holds for semimetric spaces with the following triangle functions: $\Phi(u,v)=u+v$; $\Phi(u,v)=K(u+v)$, $1\leqslant K \leqslant 2$; $\Phi(u,v)=\max\{u,v\}$; $\Phi(u,v)=(u^q+v^q)^{\frac{1}{q}}$, $q>0$, and with the corresponding estimations~(\ref{ee5}) from above for $d(x_n,x^*)$.
\end{cor}
\begin{proof}
The proof follows directly from Corollaries~\ref{c33},~\ref{c34},~\ref{c35} and from the fact that all above mentioned triangle functions satisfy condition (i) of Theorem~\ref{t7}.
\end{proof}

\section{Chatterjea's contractions in semimetric spaces}

In~\cite{Ch72} Chatterjea  proved the following result.
\begin{thm}\label{t40}
Let $T\colon X\to X$ be a mapping on a complete metric space $(X,d)$ such that
  \begin{equation}\label{Ch}
   d(Tx,Ty)\leqslant \beta (d(x,Ty)+d(y,Tx)),
  \end{equation}
where $0\leqslant \beta<\frac{1}{2}$ and $x,y \in X$. Then $T$ has a unique fixed point.
\end{thm}
The mappings satisfying inequality~(\ref{Ch}) are called \emph{Chatterjea type mappings}.

{To prove the following theorem we need the notion of an inverse function for a nondecreasing function. This is due to the fact that the aim of this theorem is also to cover the class of ultrametric spaces and the fact that the function $\Psi(u)=\max\{u,1\}$ is not strictly increasing. By~\cite[p.~34]{GRSY12} for every nondecreasing function $\Psi\colon [0,\infty]\to [0,\infty]$ the inverse function $\Psi^{-1}\colon [0,\infty]\to [0,\infty]$ can be well defined by setting}
$$
\Psi^{-1}(\tau)=\inf\limits_{\Psi(t)\geqslant \tau} t.
$$

{Here, $\inf$ is equal to $\infty$ if the set of $t \in[0,\infty]$ such that $\Psi(t)\geqslant \tau$ is empty. Note that the function  $\Psi^{-1}$ is nondecreasing too. It is evident immediately by the definition that}
\begin{equation}\label{r1}
  \Psi^{-1}(\Psi(t))\leqslant t \, \text{ for all } \, t\in [0,\infty].
\end{equation}

\begin{thm}\label{t08}
Let $(X,d)$ be a complete semimetric space with the {continuous} triangle function $\Phi$, satisfying conditions~(\ref{ee1}) and~(\ref{ee2}), and such that the semimetric $d$ is continuous. Let  $T\colon X\to X$ satisfy inequality~(\ref{Ch}) with some real number $\beta\geqslant  0$ such that the following conditions hold:
\begin{itemize}
  \item [(i)] $\Phi(0,\beta)<1$.
  \item [(ii)] $\Psi^{-1}(1/\beta)> 1$ if $\beta>0$, where $\Psi(u)=\Phi(u,1)$.
\end{itemize}
Then $T$ has a fixed point. If $0 \leqslant \beta< \frac{1}{2}$, then the fixed point is unique.
\end{thm}

\begin{proof}
Let $\beta=0$. Then~(\ref{Ch}) is equivalent to $d(Tx,Ty)=0$ for all $x,y \in X$. Let $x_0\in X$ and $x^*=Tx_0$, then $d(Tx_0,T(Tx_0))=0$ and $d(x^*,Tx^*)$=0. Hence, $x^*$ is a fixed point. Suppose that there exist another fixed point $x^{**}\neq x^*$, $x^{**}=Tx^{**}$. Then by the equality $d(Tx,Ty)=0$ we have $d(Tx^*,Tx^{**})=d(x^*,x^{**})=0$, which is a contradiction.

Let now $\beta>0$ and let $x_0 \in X$. Define $x_n = T^n x_0$ for $n = 1, 2, \ldots$. If $x_i=x_{i+1}$ for some $i$, then it is clear that $x_i$ is a fixed point. Suppose that $x_i\neq x_{i+1}$ for all $i$.

It follows straightforwardly that
\begin{align*}
    d(x_{n},x_{n+1}) &= d(Tx_{n-1},Tx_{n}) \leqslant   \beta (d(x_{n-1},Tx_{n}) + d(x_{n},Tx_{n-1})) \\
               &=  \beta (d(x_{n-1},x_{n+1}) + d(x_{n},x_{n}))= \beta d(x_{n-1},x_{n+1}).
\end{align*}
Hence, by triangle inequality~(\ref{tr}) and condition~(\ref{ee1}) we get
$$
d(x_{n},\,x_{{n+1}})\leqslant \beta \Phi(d(x_{{n-1}},\,x_{{n}}), d(x_{{n}},\,x_{{n+1}}))
$$
and
\begin{equation}\label{w1}
\frac{1}{\beta}\leqslant \Phi\left(\frac{d(x_{{n-1}},\,x_{{n}})}{d(x_{n},\,x_{{n+1}})}, 1\right)=\Psi\left(\frac{d(x_{{n-1}},\,x_{{n}})}{d(x_{n},\,x_{{n+1}})}\right),
\end{equation}
where $\Psi(u)=\Phi(u,1)$, $u\in [0,\infty)$. It is clear that $\Psi(u)$ is {nondecreasing }on $[0,\infty)$. Hence, $\Psi^{-1}(u)$ is also {nondecreasing} on $[0,\infty)$.
Hence, {it follows from~(\ref{w1}) and~(\ref{r1}) that}
$$
\Psi^{-1}\bigg(\frac{1}{\beta}\bigg)\leqslant \frac{d(x_{{n-1}},\,x_{{n}})}{d(x_{n},\,x_{{n+1}})}
$$
and
$$
d(x_{n},\,x_{{n+1}}) \leqslant \Big(\Psi^{-1}\Big({1}/{\beta}\Big)\Big)^{-1} {d(x_{{n-1}},\,x_{{n}})}.
$$
Consequently,
\[
d(x_{n},x_{n+1}) \leqslant \alpha d(x_{n-1},x_{n}),
\]
where $\alpha = \Big(\Psi^{-1}\Big({1}/{\beta}\Big)\Big)^{-1}$.
{Since by condition (ii) $\Psi^{-1}(1/\beta)> 1$ we get  $0\leqslant \alpha <1$.} By Lemma~\ref{lem} $(x_n)$ is a Cauchy sequence and  by completeness of $(X,d)$, this sequence has a limit $x^*\in X$.

Let us prove that $Tx^*=x^*$.  By triangle inequality~(\ref{tr}), the monotonicity of $\Phi$ and~(\ref{Ch}) we get
$$
d(x^*,Tx^*)\leqslant \Phi(d(x^*,T^nx_{0}),d(T^nx_{0},Tx^*))
$$
$$
\leqslant \Phi(d(x^*,T^nx_{0}),\beta (d(T^{n-1}x_{0},Tx^*)+d(x^*,T^nx_0))).
$$
Letting $n\to \infty$, by the continuity of $\Phi$ and $d$ we obtain
$$
d(x^*,Tx^*)\leqslant \Phi(0,\beta d(x^*,Tx^*)).
$$
{Using~(\ref{ee1}), we have
$$
d(x^*,Tx^*)\leqslant d(x^*,Tx^*)\Phi(0,\beta).
$$}
By condition (i) we get $d(x^*,Tx^*)=0$.

Suppose that there exist two distinct fixed points $x$ and $y$. Then $Tx=x$ and $Ty=y$, which contradicts to~(\ref{kk}) in the case $0\leqslant\beta<\frac{1}{2}$.
\end{proof}

\begin{cor}\label{c45}
Theorem~\ref{t08} holds in ultrametric spaces with the coefficient $0\leqslant \beta < 1$.
\end{cor}
\begin{proof}
According to the assumption $\Phi(u,v)=\max\{u,v\}$, $\Psi(u)=\max\{u,1\}$ and
$$
\Psi^{-1}(u)=
  \begin{cases}
    0, & u\in[0,1], \\
    u, & u\in(1,\infty).
  \end{cases}
$$
Clearly, condition (i) holds for all $0\leqslant\beta<1$ and condition (ii) holds for all $0<\beta<1$.
\end{proof}

\begin{cor}\label{c46}
Theorem~\ref{t08} holds for semimetric spaces with the following triangle functions $\Phi(u,v)=(u^q+v^q)^{\frac{1}{q}}$, $q\geqslant 1$, and with the coefficient $0\leqslant\beta<2^{-1/q}$ in~(\ref{Ch}).
\end{cor}
\begin{proof}
We have  $\Psi(u)=(u^q+1)^{\frac{1}{q}}$ and $\Psi^{-1}(u)=(u^q-1)^{\frac{1}{q}}$. Clearly, condition (i) holds for all $0\leqslant\beta<1$ but condition (ii) holds if $0<\beta<2^{-1/q}$.
\end{proof}

Note that the following proposition is already known, see Theorem 3 in~\cite{KK13}. But it does not follow from Theorem~\ref{t08} since the semimetric $d$ in a b-metric space $(X,d)$ is not obligatory continuous if $K>1$.
\begin{prop}\label{c43}
Theorem~\ref{t08} holds in b-metric spaces with $K\geqslant 1$ and with the coefficient $0\leqslant \beta<\frac{1}{2K}$ in~(\ref{Ch}).
\end{prop}

\begin{cor}\label{cx1}
Theorem~\ref{t40} holds.
\end{cor}
\begin{proof}
It suffices to set $K=1$ in Proposition~\ref{c43} or $q=1$ in Corollary~\ref{c46}.
\end{proof}

\section{\'Ciri\'c-Reich-Rus's contractions  in semimetric spaces}
In 1971, independently, \'Ciri\'c \cite{CB71}, Reich ~\cite{Re71} and Rus \cite{Ru71} extended the Kannan fixed point theorem to cover a broader class of mappings.
\begin{thm}
Let $T\colon X\to X$ be a mapping on a complete metric space $(X,d)$ with
\begin{equation}\label{CRR}
  d(Tx,Ty) \leqslant \alpha d(x,y)+\beta d(x,Tx)+\gamma d(y,Ty),
\end{equation}
$\alpha\geqslant 0, \beta\geqslant 0, \gamma\geqslant 0$ and $\alpha+\beta+\gamma <1$. Then $T$ has a unique fixed point.
\end{thm}
In what follows, we will refer to the mapping (\ref{CRR}) as the \'Ciri\'c-Reich-Rus mapping. This theorem integrates principles from both the Banach contraction principle (by choosing $\beta=\gamma=0$) and the Kannan fixed point theorem with $\alpha=0$ and $\beta=\gamma$.

\begin{thm}\label{t8}
Let $(X,d)$ be a complete semimetric space with the continuous triangle function $\Phi$, satisfying conditions~(\ref{ee1}) and~(\ref{ee2}). Let  $T\colon X\to X$ be a \'Ciri\'c-Reich-Rus mapping with the coefficients $\alpha\geqslant 0, \beta\geqslant 0, \gamma\geqslant 0$, $\alpha+\beta+\gamma <1$, and let additionally the following  condition hold:
\begin{itemize}
   \item [(i)]  $\Phi(0,\gamma)<1$.
\end{itemize}
 Then $T$ has a unique fixed point.
\end{thm}

\begin{proof}
Let $x_0\in X$. Define $x_n = T^n x_0$ for $n = 1, 2, \ldots$. Then, it follows straightforwardly that
\begin{align*}
    d(x_{n},x_{n+{1}})=d(Tx_{n-1},Tx_{n}) &\leqslant \alpha d(x_{n-1}, x_{n}) + \beta d(x_{n-1},Tx_{n-1}) + \gamma d(x_{n},Tx_{n}) \\
               &= \alpha d(x_{n-1}, x_{n}) + \beta d(x_{n-1},x_{n}) + \gamma d(x_{n},x_{n+{1}}).
\end{align*}
Hence,
$$
d(x_{n},x_{n+{1}}) \leqslant \delta d(x_{n-1},x_{n}),
$$
where $\delta = \frac{\alpha+\beta}{1-\gamma}$, $0\leqslant \delta <1$. By Lemma~\ref{lem} $(x_n)$ is a Cauchy sequence and  by completeness of $(X,d)$, this sequence has a limit $x^*\in X$.

Let us prove that $Tx^*=x^*$. 
By triangle inequality~(\ref{tr}), the monotonicity of $\Phi$ and~(\ref{CRR}) we get
$$
d(x^*,Tx^*)\leqslant \Phi(d(x^*,T^nx_{0}),d(T^nx_{0},Tx^*))
$$
$$
\leqslant \Phi(d(x^*,T^nx_{0}),\alpha d(T^{n-1}x_{0}, x^*)+\beta d(T^{n-1}x_{0},T^{n}x_{0})+\gamma d(x^*,Tx^*)).
$$
Letting $n\to \infty$, by the continuity of $\Phi$ we obtain
$$
d(x^*,Tx^*)\leqslant \Phi(0,\gamma d(x^*,Tx^*)).
$$
{Using~(\ref{ee1}), we have
$$
d(x^*,Tx^*)\leqslant d(x^*,Tx^*)\Phi(0,\gamma).
$$}
By condition (i) we get $d(x^*,Tx^*)=0$.

Suppose that there exist two distinct fixed points $x$ and $y$. Then $Tx=x$ and $Ty=y$, which contradicts to~(\ref{CRR}).
\end{proof}

\begin{cor}\label{c35x5}
Theorem~\ref{t8} holds for semimetric spaces with the following triangle functions: $\Phi(u,v)=u+v$; $\Phi(u,v)=K(u+v)$, $1\leqslant K<1/\gamma$; $\Phi(u,v)=\max\{u,v\}$; $\Phi(u,v)=(u^q+v^q)^{\frac{1}{q}}$, $q>0$, with the corresponding estimations~(\ref{ee5}) from above for $d(x_n,x^*)$.
\end{cor}

\section{Mappings contracting perimeters of triangles in semimetric spaces}

Let $X$ be a metric space. In~\cite{P23}, a new type of mappings $T\colon X\to X$ was considered, characterized as mappings contracting perimeters of triangles (see Definition~\ref{d1}). It was demonstrated that such mappings are continuous. Furthermore, a fixed-point theorem for such mappings was proven, with the classical Banach fixed-point theorem emerging as a simple corollary. An example of a mapping contracting perimeters of triangles, which is not a contraction mapping, was constructed for a space $X$ with $\operatorname{card}(X)=\aleph_0$. In this section, we establish a generalization of the aforementioned theorem.

The following definition was introduced in~\cite{P23} for the case of ordinary metric spaces. In this work, we extend it for the case of general semimetric spaces.

\begin{defn}\label{d1}
Let $(X,d)$ be a semimetric space with $|X|\geqslant 3$. We shall say that $T\colon X\to X$ is a \emph{mapping contracting perimeters of triangles} on $X$ if there exists $\alpha\in [0,1)$ such that the inequality
  \begin{equation}\label{e1}
   d(Tx,Ty)+d(Ty,Tz)+d(Tx,Tz) \leqslant \alpha (d(x,y)+d(y,z)+d(x,z))
  \end{equation}
  holds for all three pairwise distinct points $x,y,z \in X$.
\end{defn}

\begin{rem}
Note that the requirement for $x,y,z\in X$ to be pairwise distinct in Definition~\ref{d1} is essential. One can see that otherwise this definition is equivalent to the definition of contraction mapping.
\end{rem}

We shall say that $x_0$ is an accumulation point of the semimetric space $(X,d)$ if for every $\ve>0$ there exists $x\in X$, $x\neq x_0$, such that $d(x_0,x)\leqslant \ve$.\\

The subsequent proposition demonstrates that mappings contracting perimeters of triangles are continuous not only in ordinary metric spaces but also in more general semimetric spaces with triangle functions continuous at the origin.

\begin{prop}\label{p21}
Let $(X,d)$, $|X|\geqslant 3$, be a semimetric space with a triangle function $\Phi$ continuous at $(0,0)$  and let $T\colon X\to X$ be a mapping contracting perimeters of triangles on $X$. Then $T$ is continuous.
\end{prop}
\begin{proof}
Let $(X,d)$ be a metric space with $|X|\geqslant 3$, $T\colon X\to X$ be a mapping contracting perimeters of triangles on $X$ and let $x_0$ be an isolated point in $X$. Then, clearly, $T$ is continuous at $x_0$. Let now $x_0$ be an accumulation point. Let us show that for every $\ve>0$, there exists $\delta>0$ such that $d(Tx_0,Tx)<\ve$ whenever $d(x_0,x)<\delta$. Suppose that $x\neq x_0$, otherwise this assertion is evident.
Since $x_0$ is an accumulation point, for every $\delta>0$ there exists $y\in X$ such that $x_0\neq y\neq x$ and $d(x_0,y)< \delta$. Since the points $x_0$, $x$ and $y$ are pairwise distinct by~(\ref{e1}) we have
$$
d(Tx_0,Tx)\leqslant d(Tx_0,Tx)+d(Tx_0,Ty)+d(Tx,Ty)
$$
$$
\leqslant \alpha(d(x_0,x)+d(x_0,y)+d(x,y)).
$$
Using the triangle inequality $d(x,y) \leqslant \Phi(d(x_0,x),d(x_0,y))$ and monotonicity of $\Phi$, we get
$$
d(Tx_0,Tx)\leqslant \alpha(d(x_0,x)+d(x_0,y)+\Phi(d(x_0,x),d(x_0,y)))
\leqslant \alpha(2\delta+\Phi(\delta,\delta)).
$$
Since $\Phi$ is continuous at $(0,0)$ and $\Phi(0,0)=0$ we get that for every $\ve>0$ there exists $\delta>0$ such that the nequality
$ \alpha(2\delta+\Phi(\delta,\delta))<\ve$ holds, which completes the proof.
\end{proof}

Let $T$ be a mapping on the metric space $X$. A point $x\in X$ is called a \emph{periodic point of period $n$} if $T^n(x) = x$. The least positive integer $n$ for which $T^n(x) = x$ is called the prime period of $x$. In particular, the point $x$ is of prime period $2$ if $T(T(x))=x$ and $Tx\neq x$.\\

The following theorem is the main result of this section.
\begin{thm}\label{t1}
Let $(X,d)$, $|X|\geqslant 3$, be a complete semimetric space with the triangle function $\Phi$ continuous at $(0,0)$ and satisfying conditions~(\ref{ee1}) and~(\ref{ee2}) and let the mapping $T\colon X\to X$ satisfy the following two conditions:
\begin{itemize}
  \item [(i)] $T$ does not possess periodic points of prime period $2$.
  \item [(ii)] $T$ is a mapping contracting perimeters of triangles on $X$.
\end{itemize}
Then $T$ has a fixed point. The number of fixed points is at most two.
\end{thm}

\begin{proof}
Let $x_0\in X$, $Tx_0=x_1$, $Tx_1=x_2$, \ldots, $Tx_n=x_{n+1}$, \ldots. Suppose first that $x_i$ is not a fixed point of the mapping $T$ for every $i=0,1,...$. Let us show that all $x_i$ are different. Since $x_i$ is not fixed, then $x_i\neq x_{i+1}=Tx_i$. By condition (i) $x_{i+2}=T(T(x_i))\neq x_i$ and by the supposition that $x_{i+1}$ is not fixed we have $x_{i+1}\neq x_{i+2}=Tx_{i+1}$. Hence, $x_i$, $x_{i+1}$ and $x_{i+2}$ are pairwise distinct. Further, set
$$
p_0=d(x_0,x_1)+d(x_1,x_2)+d(x_2,x_0),
$$
$$
p_1=d(x_1,x_2)+d(x_2,x_3)+d(x_3,x_1),
$$
$$
\cdots
$$
$$
p_n=d(x_n,x_{n+1})+d(x_{n+1},x_{n+2})+d(x_{n+2},x_n),
$$
$$
\cdots .
$$
Since  $x_i$, $x_{i+1}$ and $x_{i+2}$ are pairwise distinct by~(\ref{e1}) we have $p_1\leqslant \alpha p_0$, $p_2\leqslant \alpha p_1$, \ldots, $p_n\leqslant \alpha p_{n-1}$ and
\begin{equation}\label{e2}
p_0>p_1>...>p_n>\ldots .
\end{equation}
Suppose now that $j\geqslant 3$ is a minimal natural number such that $x_j=x_i$ for some $i$ such that $0\leqslant i<j-2$. Then  $x_{j+1}=x_{i+1}$, $x_{j+2}=x_{i+2}$. Hence, $p_i=p_j$ which contradicts to~(\ref{e2}). Thus, all $x_i$ are different.

Further, let us show that $(x_i)$ is a Cauchy sequence.  It is clear that
$$
d(x_0,x_1)\leqslant p_0,
$$

$$
d(x_1,x_2)\leqslant p_1\leqslant \alpha p_0,
$$

$$
d(x_2,x_3)\leqslant p_2\leqslant \alpha p_1\leqslant  \alpha^2 p_0,
$$

$$
\cdots
$$

\begin{equation*}
d(x_{n-1},x_{n})\leqslant p_{n-1}\leqslant \alpha^{n-1} p_0,
\end{equation*}

\begin{equation}\label{e311}
d(x_{n},x_{n+1})\leqslant p_{n}\leqslant \alpha^{n} p_0,
\end{equation}
$$
\cdots .
$$
Comparing~(\ref{e311}) with~(\ref{e31n}) and using the proof of Lemma~\ref{lem} we get that
$(x_n)$ is a Cauchy sequence. By completeness of $(X,d)$, this sequence has a limit $x^*\in X$.

Let us prove that $Tx^*=x^*$. Since $x_n\to x^*$, by continuity of $T$ we have $x_{n+1}=T x_n\to Tx^*$.  By triangle inequality~(\ref{tr}) and continuity of $\Phi$ at $(0,0)$ we have
$$
d(x^*,Tx^*)\leqslant \Phi(d(x^*,x_{n}),d(x_{n},Tx^*))
\to 0
$$
as $n\to \infty$, which means that $x^*$ is the fixed point.

Suppose that there exist at least three pairwise distinct fixed points $x$, $y$ and $z$.  Then $Tx=x$, $Ty=y$ and $Tz=z$, which contradicts to~(\ref{e1}).
\end{proof}

\begin{cor}\label{c338}
Theorem~\ref{t1} holds for semimetric spaces with the following triangle functions: $\Phi(u,v)=u+v$; $\Phi(u,v)=K(u+v)$, $K \geqslant 1$; $\Phi(u,v)=\max\{u,v\}$; $\Phi(u,v)=(u^q+v^q)^{\frac{1}{q}}$, $q>0$, with the corresponding estimations~(\ref{ee5}) from above for $d(x_n,x^*)$.
\end{cor}

{The following example shows that Condition (i) in Theorem~\ref{t1} is necessary.}
\begin{ex}
{Let us construct an example of the mapping $T$ contracting perimeters of triangles which does not have any fixed point. Let $X=\{x,y,z\}$, $d(x,y)=d(y,z)=d(x,z)=1$, and let $T\colon X\to X$ be such that $Tx=y$, $Ty=x$ and $Tz=x$.
In this case the points $x$ and $y$ are periodic points of prime period 2.
}\end{ex}

\section{Applications}

Fixed point theorems offer a robust framework for comprehending and addressing the solutions to linear and nonlinear problems that arise in biological, engineering, and physical sciences.

In Chapter 6 of Subramaniyam's monograph \cite{Su18}, various applications of the contraction principle are explored. These applications span domains including Fredholm and Volterra integral equations, existence theorems for initial value problems of first-order ordinary differential equations (ODEs), solutions of second-order ODE boundary value problems (BVPs), functional differential equations, discrete BVPs, a variety of functional equations, commutative algebra and  fractals (see also \cite{Ki01},\cite{AJS18}, \cite{Ma75}, \cite{KS13} and references therein).

In its multifaceted nature, fixed point theorems play a pivotal role in analyzing solutions to nonlinear partial differential equations (PDEs). Notably, Brouwer's, Schauder's, and Schaefer's fixed point theorems, among others, have emerged as powerful tools for ensuring the existence and uniqueness of solutions across a diverse spectrum of nonlinear PDEs (see ~\cite{Al19}, \cite{AH76} and references therein).

\section{Conclusion and future research directions}
In summary, our paper has revisited numerous renowned fixed-point theorems, providing extensions by adjusting assumptions and introducing innovative proofs. Utilizing Lemma~\ref{lem} and its corollary, we have gained further insights into the essence of fixed-point theorems, broadening their relevance beyond metric spaces to encompass more general scenarios. This investigation indicates promising directions for future research, especially concerning the application of our approach to other contractive mappings across diverse conditions. Additionally, exploring real-world applications in light of established results offers intriguing possibilities for addressing various practical problems across different settings.

\section*{Conflict of Interest Statement}

The authors declare that the research was conducted in the absence of any commercial or financial relationships that could be construed as a potential conflict of interest.

\section*{Author Contributions}


All authors listed have made a substantial, direct, and intellectual contribution to the work and approved it for publication.

\section*{Funding}

The first author was partially supported by the Volkswagen Foundation grant within the frameworks of the international project ``From Modeling and Analysis to Approximation''.
This work was also partially supported by a grant from the Simons Foundation (Award 1160640, Presidential Discretionary-Ukraine Support Grants, E. Petrov).


\section*{Data Availability Statement}

Data sharing is not applicable to this article as no new data were created or analyzed in this study.



\end{document}